\documentclass[10pt]{article}
\usepackage{amsmath, amsthm,amssymb}
\usepackage[english]{babel}
\usepackage[latin1]{inputenc}
\usepackage{array}
\newtheorem{theorem}{Theorem}

\newtheorem{definition}[theorem]{Definition}
\newtheorem{proposition}[theorem]{Proposition}
\newtheorem{lemma}[theorem]{Lemma}

\begin{document}

\title{Incomplete Generalized Fibonacci and Lucas Polynomials}

\author{José L. Ramírez$^{1}$\\
$^1$Instituto de Matemáticas y sus Aplicaciones.  \\ 110221, Universidad Sergio Arboleda,\\   Colombia\\
josel.ramirez@ima.usergioarboleda.edu.co\\[2pt]
}

\maketitle

\begin{abstract}
In this paper, we define the incomplete $h(x)$-Fibonacci and $h(x)$-Lucas polynomials, we study the recurrence relations and some properties of these polynomials.
 

\end{abstract}

\section{Introduction}
Fibonacci numbers and their generalizations have many interesting properties and applications to almost every fields of science and art  (e.g., see \cite{koshy}). The Fibonacci numbers $F_n$ are the terms of the sequence ${0, 1, 1, 2, 3, 5, . . .}$, wherein each term is the sum of the two previous terms, beginning with the values $F_0 = 0$ and $F_1 = 1$.\\

The incomplete Fibonacci and Lucas numbers were introduced by Filipponi \cite{FILI}. The incomplete Fibonacci numbers $F_n(k)$ and the incomplete Lucas numbers $L_n(k)$  are defined by

\begin{align*}
F_n(k)=\sum_{j=0}^{k}\binom{n-1-j}{j}  \ \ \left(n=1, 2, 3, \ldots; 0 \leq k \leq \left\lfloor\frac{n-1}{2}\right\rfloor\right),
\end{align*}
and
\begin{align*}
L_n(k)=\sum_{j=0}^{k}\frac{n}{n-j}\binom{n-j}{j}  \ \ \left(n=1, 2, 3, \ldots; 0 \leq k \leq \left\lfloor\frac{n}{2}\right\rfloor\right).
\end{align*}
Is is easily seen that
\begin{align*}
F_n\left(\left\lfloor\frac{n-1}{2}\right\rfloor\right)=F_n  \ \   \text{and} \  \ L_n\left(\left\lfloor\frac{n}{2}\right\rfloor\right)=L_n,
\end{align*}
where $F_n$ and $L_n$ are the $n$-th Fibonacci and Lucas numbers \cite{koshy}. Further in \cite{PINTER}, generating functions of the incomplete Fibonacci and Lucas numbers are determined. In \cite{djor}  Djordevi\'{c} gave the incomplete generalized Fibonacci and Lucas numbers.  In \cite{djor2} Djordevi\'{c}  and Srivastava defined incomplete generalized Jacobsthal and Jacobsthal-Lucas numbers. In \cite{durs}  the authors define the incomplete Fibonacci and Lucas $p$-numbers. Also the authors define the incomplete bivariate Fibonacci and Lucas $p$-polynomials in \cite{durs2}.\\

On the other hand, large classes of polynomials can be defined by Fibonacci-like recurrence relation and yield Fibonacci numbers \cite{koshy}.  Such polynomials, called Fibonacci polynomials,
were studied in 1883 by the Belgian mathematician Eugene Charles Catalan and the German mathematician E. Jacobsthal.
The polynomials $F_n(x)$ studied by Catalan are defined by the recurrence relation
\begin{align*}
F_{0}(x)=0, \  \ F_{1}(x)=1,  \   \  F_{n+1}(x)=xF_{n}(x)+F_{n-1}(x),  \ n\geqslant 1. %\label{eq11}
\end{align*}
The Fibonacci polynomials studied by Jacobsthal are defined by
\begin{align*}
J_{0}(x)=1, \  \ J_{1}(x)=1,  \   \  J_{n+1}(x)=J_{n}(x)+xJ_{n-1}(x),  \ n\geqslant 1. %\label{eq12}
\end{align*}
The Lucas polynomials $L_n(x)$, originally studied in 1970 by Bicknell, are defined by
\begin{align*}
L_{0}(x)=2, \  \ L_{1}(x)=x,  \   \  L_{n+1}(x)=xL_{n}(x)+L_{n-1}(x),  \ n\geqslant 1. %\label{eq13}
\end{align*}
In \cite{ayse}, the authors introduced the $h(x)$-Fibonacci polynomials. That generalize
Catalan's Fibonacci polynomials $F_n(x)$ and the $k$-Fibonacci numbers $F_{k,n}$ \cite{falcon1}.   Let $h(x)$ be a polynomial with real coefficients. The $h(x)$-Fibonacci polynomials $\{F_{h,n}(x)\}_{n\in \mathbb{N}}$ are defined by the recurrence relation
\begin{align}
F_{h,0}(x)=0, \  \ F_{h,1}(x)=1,  \   \  F_{h,n+1}(x)=h(x)F_{h,n}(x)+F_{h,n-1}(x),  \ n\geqslant 1. \label{eq22}
\end{align}
For $h(x)=x$ we obtain Catalan's Fibonacci polynomials, and for $h(x)=k$ we obtain $k$-Fibonacci numbers.  For $k=1$ and $k=2$ we obtain the usual Fibonacci numbers and the Pell numbers.

Let $h(x)$ be a polynomial with real coefficients. The $h(x)$-Lucas polynomials $\{L_{h,n}(x)\}_{n\in \mathbb{N}}$ are defined by the recurrence relation
\begin{align*}
L_{h,0}(x)=2, \  \ L_{h,1}(x)=h(x),  \   \  L_{h,n+1}(x)=h(x)L_{h,n}(x)+L_{h,n-1}(x),  \ n\geqslant 1. %\label{eq33}
\end{align*}
For  $h(x)=x$ we obtain the Lucas polynomials, and for $h(x)=k$ we have the $k$-Lucas  numbers \cite{falcon3}. For $k=1$ we obtain the usual  Lucas  numbers.  In \cite{ayse},  the authors obtained some relations for these polynomials sequences. In addition, in \cite{ayse}, the explicit formula to  $h(x)$-Fibonacci polynomials is
  \begin{align}
F_{h,n}(x)=\sum_{i=0}^{\left\lfloor\frac{n-1}{2}\right\rfloor}\binom{n-i-1}{i}h^{n-2i-1}(x),  \label{ec20}
\end{align}
and the explicit formula of $h(x)$-Lucas polynomials is
  \begin{align*}
L_{h,n}(x)=\sum_{i=0}^{\left\lfloor\frac{n}{2}\right\rfloor}\frac{n}{n-i}\binom{n-i}{i}h^{n-2i}(x).
\end{align*}

In this paper, we introduce the incomplete $h(x)$-Fibonacci and $h(x)$-Lucas polynomials and we obtain new identities.

\section{Some Properties of $h(x)$-Fibonacci and $h(x)$-Lucas Polynomials}

The characteristic equation associated with the recurrence relation (\ref{eq22}) is $v^2=h(x)v+1$.  The roots of  this equation are
\begin{align*}
\alpha(x)=\frac{h(x)+\sqrt{h(x)^2+4}}{2},    \hspace{1cm}   \beta(x)=\frac{h(x)-\sqrt{h(x)^2+4}}{2}.
\end{align*}
Then we have the following basic identities:
\begin{align} \label{iden}
\alpha(x)+\beta(x)=h(x),\  \ \hspace{0.5cm}  \alpha(x)-\beta(x)=\sqrt{h(x)^2+4},\  \ \hspace{0.5cm} \  \  \alpha(x)\beta(x)=-1.
\end{align}
Some of the properties that the $h(x)$-Fibonacci polynomials verify are summarized bellow (see \cite{ayse} for the proofs).

\begin{itemize}
\item Binet formula: $F_{h,n}(x)=\frac{\alpha(x)^n-\beta(x)^n}{\alpha(x)-\beta(x)}$ .
\item Combinatorial formula: $F_{h,n}(x)=\sum_{i=0}^{\lfloor (n-1)/2 \rfloor}\binom{n-1-i}{i}h^{n-1-2i}(x)$.
\item Generating function: $g_f(t)=\frac{t}{1-h(x)t-t^2}$.
\end{itemize}

Some properties that the $h(x)$-Lucas numbers verify are summarized bellow (see \cite{ayse} for the proofs).
\begin{itemize}
\item Binet formula: $L_{h,n}(x)=\alpha(x)^n+\beta(x)^n$.
\item Relation with $h(x)$-Fibonacci polynomials: $L_{h,n}(x)=F_{h,n-1}(x)+F_{h,n+1}(x), \ n \geqslant 1$.
\end{itemize}

\section{The incomplete $h(x)$-Fibonacci Polynomials}

\subsection{Definition}

\begin{definition}\label{def1}
The incomplete $h(x)$-Fibonacci polynomials are defined by
  \begin{align}
F_{h,n}^l(x)=\sum_{i=0}^{l}\binom{n-1-i}{i}h^{n-2i-1}(x), \ \ 0\leq l \leq \left\lfloor \frac{n-1}{2}\right\rfloor .  \label{ec1}
\end{align}
\end{definition}

In Table \ref{table1}, some polynomials of incomplete $h(x)$-Fibonacci polynomials are provided.
\begin{table}[h]\tiny
\centering
\begin{tabular}{|>{$}c<{$}|>{$}c<{$}| >{$}c<{$}|>{$}c<{$}|>{$}c<{$}|>{$}c<{$}| }\hline
 n \setminus l & 0 & 1 & 2 & 3 & 4 \\ \hline
 1 & 1 & \text{} & \text{} & \text{} & \text{} \\
 2 & h & \text{} & \text{} & \text{} & \text{} \\
 3 & h^2 & h^2+1 & \text{} & \text{} & \text{} \\
 4 & h^3 & h^3+2 h & \text{} & \text{} & \text{} \\
 5 & h^4 & h^4+3 h^2 & h^4+3 h^2+1 & \text{} & \text{} \\
 6 & h^5 & h^5+4 h^3 & h^5+4 h^3+3 h & \text{} & \text{} \\
 7 & h^6 & h^6+5 h^4 & h^6+5 h^4+6 h^2 & h^6+5 h^4+6 h^2+1 & \text{} \\
 8 & h^7 & h^7+6 h^5 & h^7+6 h^5+10 h^3 & h^7+6 h^5+10 h^3+4 h & \text{} \\
 9 & h^8 & h^8+7 h^6 & h^8+7 h^6+15 h^4 & h^8+7 h^6+15 h^4+10 h^2 & h^8+7 h^6+15 h^4+10 h^2+1 \\
 10 & h^9 & h^9+8 h^7 & h^9+8 h^7+21 h^5 & h^9+8 h^7+21 h^5+20 h^3 & h^9+8 h^7+21 h^5+20 h^3+5 h \\ \hline \end{tabular}
\caption{The polynomials $F_{h, n}^{l}(x)$ for $1\leqslant n \leqslant 10$} \label{table1}
\end{table}

We note that
\begin{align*}
F_{1,n}^{ \left\lfloor \frac{n-1}{2}\right\rfloor}(x)=F_n.
\end{align*}

For $h(x)=1$,  we get incomplete Fibonacci numbers \cite{FILI}. If $h(x)=k$ we obtained incomplete $k$-Fibonacci numbers.

Some special cases of (\ref{ec1}) are
\begin{align}
&F_{h,n}^0(x)=h^{n-1}(x); \ (n\geq1)\\
&F_{h,n}^1(x)=h^{n-1}(x)+(n-2)h^{n-3}(x); \ (n\geq3)\\
&F_{h,n}^2(x)=h^{n-1}(x)+(n-2)h^{n-3}(x) + \frac{(n-4)(n-3)}{2}h^{n-5}(x); \ (n\geq5)\\
&F_{h,n}^{ \left\lfloor \frac{n-1}{2}\right\rfloor}(x)=F_{h,n}(x); \ (n\geq1)\\
&F_{h,n}^{ \left\lfloor \frac{n-3}{2} \right\rfloor}(x)=\begin{cases}
F_{h,n}(x)  - \frac{nh(x)}{2}  &  \ \text{($n$  even)} \\
F_{h,n}(x)  - 1  &  \ \text{($n$  odd)}
\end{cases}  (n\geq 3).
\end{align}

\subsection{Some recurrence properties of the polynomials $F_{h,n}^l(x)$ }

\begin{proposition}
The recurrence relation of the incomplete $h(x)$-Fibonacci polynomials $F_{h,n}^l(x)$ is
\begin{align}
F_{h,n+2}^{l+1}(x)=h(x)F_{h,n+1}^{l+1}(x) + F_{h,n}^{l}(x) \ \ 0\leq l \leq \frac{n-2}{2}  \label{ec3}
\end{align}
The relation (\ref{ec3}) can be transformed into the non-homogeneous recurrence relation
\begin{align}
F_{h,n+2}^{l}(x)=h(x)F_{h,n+1}^{l}(x) + F_{h,n}^{l}(x) - \binom{n-1-l}{l}h^{n-1-2l}(x). \label{ec15}
\end{align}
\end{proposition}
\begin{proof}
Use the Definition \ref{def1} to rewrite the right-hand side of (\ref{ec3}) as
\begin{align*}
h(x)\sum_{i=0}^{l+1}\binom{n-i}{i}h^{n-2i}(x) + &  \sum_{i=0}^{l}\binom{n-i-1}{i}h^{n-2i-1}(x)\\
&=\sum_{i=0}^{l+1}\binom{n-i}{i}h^{n-2i+1}(x) + \sum_{i=1}^{l+1}\binom{n-i}{i-1}h^{n-2i+1}(x)\\
&=h^{n-2i+1}(x)\left(\sum_{i=0}^{l+1}\left[\binom{n-i}{i} + \binom{n-i}{i-1}\right] \right) - h^{n+1}(x)\binom{n}{-1}\\
&=\sum_{i=0}^{l+1}\binom{n-i+1}{i}h^{n-2i+1}(x)-0\\
&=F_{h,n+2}^{l}(x).
\end{align*}
\end{proof}

\begin{proposition}
\begin{align}
\sum_{i=0}^{s}\binom{s}{i}F_{h,n+i}^{l+i}(x)h^i(x)=F_{h,n+2s}^{l+s}(x) \ \ \left( 0\leq l \leq \frac{n-s-1}{2} \right).  \label{ec4}
\end{align}
\end{proposition}
\begin{proof}
(By induction on $s$). The sum (\ref{ec4}) clearly holds for $s=0$ and $s=1$ (see (\ref{ec3})). Now suppose that the result is true for all $j<s+1$, we prove it for  $s+1$.
\begin{align*}
&\sum_{i=0}^{s+1}\binom{s+1}{i}F_{h,n+i}^{l+i}(x)h^i(x)\\
&=\sum_{i=0}^{s+1}\left[\binom{s}{i}+\binom{s}{i-1}\right]F_{h,n+i}^{l+i}(x)h^i(x)\\
&=\sum_{i=0}^{s+1}\binom{s}{i}F_{h,n+i}^{l+i}(x)h^i(x)+\sum_{i=0}^{s+1}\binom{s}{i-1}F_{h,n+i}^{l+i}(x)h^i(x)\\
&=F_{h,n+2s}^{l+s}(x) + \binom{s}{s+1}F_{h,n+s+1}^{l+s+1}(x)h^{s+1}(x) + \sum_{i=-1}^{s}\binom{s}{i}F_{h,n+i+1}^{l+i+1}(x)h^{i+1}(x)\\
&=F_{h,n+2s}^{l+s}(x)+0+\sum_{i=0}^{s}\binom{s}{i}F_{h,n+i+1}^{l+i+1}(x)h^{i+1}(x) + \binom{s}{-1}F_{h,n}^{l}(x)\\
&=F_{h,n+2s}^{l+s}(x)+h(x)\sum_{i=0}^{s}\binom{s}{i}F_{h,n+i+1}^{l+i+1}(x)h^{i}(x) + 0 \\
&=F_{h,n+2s}^{l+s}(x)+h(x)F_{h,n+2s+1}^{l+s+1}(x)\\
&=F_{h,n+2s+2}^{l+s+1}(x).
\end{align*}
\end{proof}

\begin{proposition}
For $n\geq 2l+2$,
\begin{align}
\sum_{i=0}^{s-1}F_{h,n+i}^{l}(x)h^{s-1-i}(x)=F_{h,n+s+1}^{l+1}(x)-h^s(x)F_{h,n+1}^{l+1}(x).  \label{ec5}
\end{align}
\end{proposition}
\begin{proof}
(By induction on $s$).  La sum (\ref{ec5}) clearly holds for $s=1$ (see (\ref{ec3})). Now suppose that the result is true for all $j<s$. We prove it for $s$.
\begin{align*}
\sum_{i=0}^{s}F_{h,n+i}^{l}(x)h^{s-i}(x)&=h(x)\sum_{i=0}^{s-1}F_{h,n+i}^{l}(x)h^{s-i-1}(x) + F_{h,n+s}^l(x)\\
&=h(x)\left(F_{h,n+s+1}^{l+1}(x)-h^{s}(x)F_{h,n+1}^{l+1}(x)\right) + F_{h,n+s}^l(x)\\
&=\left(h(x)F_{h,n+s+1}^{l+1}(x)+F_{h,n+s}^{l}(x)\right) - h^{s+1}(x)F_{h,n+1}^{l+1}(x)\\
&=F_{h,n+s+2}^{l+1}(x) - h^{s+1}(x)F_{h,n+1}^{l+1}(x).
\end{align*}
\end{proof}

\begin{lemma}\label{lem1}
\begin{align}
F'_{h,n}(x)=h'(x)\left(\frac{nL_{h,n}(x)-h(x)F_{h,n}(x)}{h^2(x)+4}\right)  \label{ec6}
\end{align}
\end{lemma}
\begin{proof}
By deriving into the Binet's formula it is obtained:
\begin{align*}
F'_{h,n}(x)=\frac{n\left[\alpha^{n-1}(x)-(-\alpha(x))^{-n-1}\right]\alpha'(x)}{\alpha(x)+\alpha(x)^{-1}}-
\frac{\left[\alpha^n(x)-(-\alpha(x))^{-n}\right](1-\alpha^{-2}(x))\alpha'(x)}{\left[\alpha(x)+\alpha^{-1}(x)\right]^2}
\end{align*}
being  $\alpha(x)=\frac{h(x)+\sqrt{h^2(x)+4}}{2}$, and therefore $\alpha'(x)=\frac{h'(x)\alpha(x)}{\alpha(x)+\alpha^{-1}(x)}$, $1-\alpha^{-2}(x)=\frac{h(x)}{\alpha(x)}$, and then
\begin{align*}
F'_{h,n}(x)=\frac{n\left[\alpha^{n}(x)+(-\alpha(x))^{-n}\right]h'(x)}{\left[\alpha(x)+\alpha^{-1}(x)\right]^2}-
\frac{\left[\alpha^n(x)-(-\alpha(x))^{-n}\right]}{\alpha(x)+\alpha^{-1}(x)}\cdot\frac{h(x)h'(x)}{\left[\alpha(x)+\alpha^{-1}(x)\right]^2}
\end{align*}
On the other hand, $F_{h,n+1}(x)+F_{h,n-1}(x)=\alpha^n(x)+\beta^n(x)=\alpha^n(x)+(-\alpha(x))^{-n}=L_{h,n}(x)$.\\
From where, after some algebra Eq. (\ref{ec6}) is obtained.
\end{proof}

Lemma \ref{lem1} generalizes Proposition 13 of \cite{falcon5}.
\begin{lemma}\label{lem2}
\begin{align}
\sum_{i=0}^{\left\lfloor \frac{n-1}{2}\right\rfloor}i\binom{n-1-i}{i}h^{n-1-2i}(x)=\frac{((h(x)^2+4)n-4)F_{h,n}(x)-nh(x)L_{h,n}(x)}{2(h^2(x)+4)}  \label{ec7}
\end{align}
\end{lemma}
\begin{proof}
From Eq(\ref{ec20}) we have that
\begin{align*}
h(x)F_{h,n}(x)=\sum_{i=0}^{\left\lfloor\frac{n-1}{2}\right\rfloor}\binom{n-1-i}{i}h^{n-2i}(x)
\end{align*}
By deriving into the above equation  it is obtained:
\begin{align*}
h'(x)F_{h,n}(x)+h(x)F'_{h,n}(x)&=\sum_{i=0}^{\left\lfloor\frac{n-1}{2}\right\rfloor}(n-2i)\binom{n-1-i}{i}h^{n-2i-1}(x)h'(x)\\
&=nF_{h,n}(x)h'(x)-2\sum_{i=0}^{\left\lfloor\frac{n-1}{2}\right\rfloor}i\binom{n-1-i}{i}h^{n-2i-1}(x)h'(x)
\end{align*}
From Lemma \ref{lem1}
\begin{multline}
h'(x)F_{h,n}(x)+h(x)h'(x)\left(\frac{nL_{h,n}(x)-h(x)F_{h,n}(x)}{h^2(x)+4}\right)\\
=nF_{h,n}(x)h'(x)-2\sum_{i=0}^{\left\lfloor\frac{n-1}{2}\right\rfloor}i\binom{n-1-i}{i}h^{n-2i-1}(x)h'(x)
\end{multline}
From where, after some algebra Eq.(\ref{ec7}) is obtained.

\end{proof}

\begin{proposition}\label{propofibo}
\begin{align}
\sum_{l=0}^{\left\lfloor\frac{n-1}{2}\right\rfloor}F_{h,n}^{l}(x)=
\begin{cases}
\frac{4F_{h,n}(x)+nh(x)L_{h,n}(x)}{2(h^2(x)+4)} \ \ & ( $n$ \ \emph{even}) \\
\frac{(h^2(x)+8)F_{h,n}(x)+nh(x)L_{h,n}(x)}{2(h^2(x)+4)} \ \ & ($n$ \ \emph{odd})  \label{ec8}
\end{cases}
 \end{align}
\end{proposition}
\begin{proof}
\begin{align*}
\sum_{l=0}^{\left\lfloor\frac{n-1}{2}\right\rfloor}F_{h,n}^{l}(x)&=F_{h,n}^{0}(x)+F_{h,n}^{1}(x)+\cdots+F_{h,n}^{\left\lfloor\frac{n-1}{2}\right\rfloor}(x)\\
&=\binom{n-1-0}{0}h^{n-1}(x) + \left[ \binom{n-1-0}{0}h^{n-1}(x) + \binom{n-1-1}{1}h^{n-3}(x)\right] + \cdots \\
&+ \left[\binom{n-1-0}{0}h^{n-1}(x) + \binom{n-1-1}{1}h^{n-3}(x)+ \cdots + \binom{n-1-\left\lfloor\frac{n-1}{2}\right\rfloor}{\left\lfloor\frac{n-1}{2}\right\rfloor}h^{n-1-2\left\lfloor\frac{n-1}{2}\right\rfloor}(x) \right]\\
&=\left(\left\lfloor\frac{n-1}{2}\right\rfloor+1\right)\binom{n-1-0}{0}h^{n-1}(x)+  \left\lfloor\frac{n-1}{2}\right\rfloor \binom{n-1-1}{1}h^{n-3}(x)+\\
&\cdots+ \binom{n-1-\left\lfloor\frac{n-1}{2}\right\rfloor}{\left\lfloor\frac{n-1}{2}\right\rfloor}h^{n-1-2\left\lfloor\frac{n-1}{2}\right\rfloor}(x)\\
&=\sum_{i=0}^{\left\lfloor\frac{n-1}{2}\right\rfloor}\left(\left\lfloor\frac{n-1}{2}\right\rfloor+1-i\right)\binom{n-1-i}{i}h^{n-1-2i}(x)\\
&=\sum_{i=0}^{\left\lfloor\frac{n-1}{2}\right\rfloor}\left(\left\lfloor\frac{n-1}{2}\right\rfloor+1\right)\binom{n-1-i}{i}h^{n-1-2i}(x) -
\sum_{i=0}^{\left\lfloor\frac{n-1}{2}\right\rfloor}i\binom{n-1-i}{i}h^{n-1-2i}(x)\\
&=\left(\left\lfloor\frac{n-1}{2}\right\rfloor+1\right)F_{h,n}(x) -
\sum_{i=0}^{\left\lfloor\frac{n-1}{2}\right\rfloor}i\binom{n-1-i}{i}h^{n-1-2i}(x)
\end{align*}
From Lemma \ref{lem2} the Eq.(\ref{ec8})  is obtained.
\end{proof}

\section{The incomplete $h(x)$-Lucas Polynomials}
\subsection{Definition}

\begin{definition}\label{def2}
The incomplete $h(x)$-Lucas polynomials are defined by
  \begin{align}
L_{h,n}^l(x)=\sum_{i=0}^{l}\frac{n}{n-i}\binom{n-i}{i}h^{n-2i}(x), \ \ 0\leq l \leq \left\lfloor \frac{n}{2}\right\rfloor. \label{ec2}
\end{align}
\end{definition}
In Table \ref{table2}, some polynomials of incomplete $h(x)$-Lucas polynomials are provided.
\begin{table}[h]\tiny
\centering
\begin{tabular}{|>{$}c<{$}|>{$}c<{$}| >{$}c<{$}|>{$}c<{$}|>{$}c<{$}|>{$}c<{$}| }\hline
 n \setminus l & 0 & 1 & 2 & 3 & 4 \\ \hline
 1 & h & \text{} & \text{} & \text{} & \text{} \\
 2 & h^2 & h^2+2 & \text{} & \text{} & \text{} \\
 3 & h^3 & h^3+3 h & \text{} & \text{} & \text{} \\
 4 & h^4 & h^4+4 h^2 & h^4+4 h^2+2 & \text{} & \text{} \\
 5 & h^5 & h^5+5 h^3 & h^5+5 h^3+5 h & \text{} & \text{} \\
 6 & h^6 & h^6+6 h^4 & h^6+6 h^4+9 h^2 & h^6+6 h^4+9 h^2+2 & \text{} \\
 7 & h^7 & h^7+7 h^5 & h^7+7 h^5+14 h^3 & h^7+7 h^5+14 h^3+7 h & \text{} \\
 8 & h^8 & h^8+8 h^6 & h^8+8 h^6+20 h^4 & h^8+8 h^6+20 h^4+16 h^2 & h^8+8 h^6+20 h^4+16 h^2+2 \\
 9 & h^9 & h^9+9 h^7 & h^9+9 h^7+27 h^5 & h^9+9 h^7+27 h^5+30 h^3 & h^9+9 h^7+27 h^5+30 h^3+9 h \\  \hline \end{tabular}
\caption{The polynomials $L_{h, n}^{l}(x)$ for $1\leqslant n \leqslant 9$} \label{table2}
\end{table}

We note that
\begin{align*}
L_{1,n}^{\left\lfloor\frac{n}{2}\right\rfloor}(x)=L_n.
\end{align*}
Some special cases of (\ref{ec2}) are
\begin{align}
&L_{h,n}^0(x)=h^{n}(x); \ (n\geq1)\\
&L_{h,n}^1(x)=h^{n}(x)+ nh^{n-2}(x); \ (n\geq2)\\
&L_{h,n}^2(x)=h^{n}(x)+ nh^{n-2}(x) + \frac{n(n-3)}{2}h^{n-4}(x); \ (n\geq4)\\
&L_{h,n}^{\left\lfloor \frac{n}{2}\right\rfloor}(x)=L_{h,n}(x); \ (n\geq1)\\
&L_{h,n}^{ \left\lfloor \frac{n-2}{2} \right\rfloor}(x)=\begin{cases}
L_{h,n}(x)  - 2   &  \ \text{($n$  even)} \\
L_{h,n}(x)  - nh(x)  &  \ \text{($n$  odd)}
\end{cases}  (n\geq 2).
\end{align}

\subsection{Some recurrence properties of the polynomials $L_{h,n}^l(x)$}

\begin{proposition}
\begin{align}
L_{h,n}^{l}(x)=F_{h,n-1}^{l-1}(x)+F_{h,n+1}^{l}(x); \ \ \left(0\leq l \leq \left\lfloor\frac{n}{2}\right\rfloor \right).  \label{ec9}
\end{align}
\end{proposition}
\begin{proof}
By (\ref{ec1}), rewrite the right-hand side of (\ref{ec9}) as
\begin{align*}
\sum_{i=0}^{l-1}\binom{n-2-i}{i}h^{n-2-2i}(x) + \sum_{i=0}^{l}\binom{n-i}{i}h^{n-2i}(x)&=
\sum_{i=1}^{l}\binom{n-1-i}{i-1}h^{n-2i}(x) + \sum_{i=0}^{l}\binom{n-i}{i}h^{n-2i}(x)\\
&=\sum_{i=0}^{l}\left[\binom{n-1-i}{i-1}+\binom{n-i}{i}\right]h^{n-2i}(x) - \binom{n-1}{-1}\\
&=\sum_{i=0}^{l}\frac{n}{n-i}\binom{n-i}{i}h^{n-2i}(x)+0\\
&=L_{h,n}^l(x).
\end{align*}
\end{proof}

\begin{proposition}
The recurrence relation of the incomplete $h(x)$-Lucas polynomials $L_{h,n}^l(x)$ is
\begin{align}
L_{h,n+2}^{l+1}(x)=h(x)L_{h,n+1}^{l+1}(x) + L_{h,n}^{l}(x) \ \ 0\leq l \leq \left\lfloor\frac{n}{2}\right\rfloor  \label{ec10}
\end{align}
The relation (\ref{ec10}) can be transformed into the non-homogeneous recurrence relation
\begin{align}
L_{h,n+2}^{l}(x)=h(x)L_{h,n+1}^{l}(x) + L_{h,n}^{l}(x) - \frac{n}{n-l}\binom{n-l}{l}h^{n-2l}(x). \label{ec16}
\end{align}
\end{proposition}
\begin{proof}
Using (\ref{ec9}) and (\ref{ec3}) we write
\begin{align*}
L_{h,n+2}^{l+1}(x)&=F_{h,n+1}^{l}(x) + F_{h,n+3}^{l+1}(x)\\
&=h(x)F_{h,n}^{l}(x) +F_{h,n-1}^{l-1}(x) + h(x)F_{h,n+2}^{l+1}(x)+ F_{h,n+1}^{l}(x)\\
&=h(x)\left(F_{h,n}^{l}(x) + F_{h,n+2}^{l+1}(x)\right) + F_{h,n-1}^{l-1}(x)+ F_{h,n+1}^{l}(x)\\
&=h(x)L_{h,n+1}^{l+1}(x) + L_{h,n}^{l}(x).
\end{align*}
\end{proof}

\begin{proposition}
\begin{align}
h(x)L_{h,n}^{l}(x)=F_{h,n+2}^{l}(x) - F_{h,n-2}^{l-2}(x) \ \ 0\leq l \leq \left\lfloor\frac{n-1}{2}\right\rfloor.  \label{ec11}
\end{align}
\end{proposition}
\begin{proof}
By (\ref{ec9}),
\begin{align*}
F_{h,n+2}^{l}(x)=L_{h,n+1}^{l}(x) - F_{h,n}^{l-1}(x) \ \ \text{and}  \ \  F_{h,n-2}^{l-2}(x)=L_{h,n-1}^{l-1}(x) - F_{h,n}^{l-1}(x),
\end{align*}
whence, from (\ref{ec10}),
\begin{align*}
F_{h,n+2}^{l}(x)-F_{h,n-2}^{l-2}(x)=L_{h,n+1}^{l}(x) - L_{h,n-1}^{l-1}(x)=h(x)L_{h,n}^{l}(x).
\end{align*}
\end{proof}

\begin{proposition}
\begin{align}
\sum_{i=0}^{s}\binom{s}{i}L_{h,n+i}^{l+i}(x)h^i(x)=L_{h,n+2s}^{l+s}(x) \ \ \left( 0\leq l \leq \frac{n-s}{2} \right).  \label{ec11}
\end{align}
\end{proposition}
\begin{proof}
Using (\ref{ec9}) and (\ref{ec4}) we write
\begin{align*}
\sum_{i=0}^{s}\binom{s}{i}L_{h,n+i}^{l+i}(x)h^i(x)&=\sum_{i=0}^{s}\binom{s}{i}\left[F_{h,n+i-1}^{l+i-1}(x)+F_{h,n+i+1}^{l+i}(x)\right]h^i(x)\\
&=\sum_{i=0}^{s}\binom{s}{i}F_{h,n+i-1}^{l+i-1}(x)h^i(x)+\sum_{i=0}^{s}\binom{s}{i}F_{h,n+i+1}^{l+i}(x)h^i(x)\\
&=F_{h,n-1+2s}^{l-1+s}(x)+F_{h,n+1+2s}^{l+s}(x)=L_{h,n+2s}^{l+s}(x).
\end{align*}
\end{proof}

\begin{proposition}
For $n\geq 2l+1$,
 \begin{align}
\sum_{i=0}^{s-1}L_{h,n+i}^{l}(x)h^{s-1-i}(x)=L_{h,n+s+1}^{l+1}(x)-h^s(x)L_{h,n+1}^{l+1}(x).  \label{ec13}
\end{align}
\end{proposition}
The proof can be done by using (\ref{ec10}) and induction on $s$.

\begin{lemma}\label{lem3}
\begin{align}
\sum_{i=0}^{\left\lfloor \frac{n}{2}\right\rfloor}i\frac{n}{n-i}\binom{n-i}{i}h^{n-2i}(x)=\frac{n}{2}\left[L_{h,n}(x)-h(x)F_{h,n}(x)\right]  \label{ec12}
\end{align}
\end{lemma}
The proof is similar to Lemma \ref{lem2}.

\begin{proposition}
\begin{align}
\sum_{l=0}^{\left\lfloor\frac{n}{2}\right\rfloor}L_{h,n}^{l}(x)=
\begin{cases}
L_{h,n}(x)+\frac{nh(x)}{2}F_{h,n}(x) \ \ & ( $n$ \ \emph{even}) \\
\frac{1}{2}\left(L_{h,n}(x)+nh(x)F_{h,n}(x)\right) \ \ & ($n$ \ \emph{odd}).  \label{ec14}
\end{cases}
 \end{align}
\end{proposition}
\begin{proof}
An argument analogous to that of the proof of Proposition \ref{propofibo} yields

\begin{align*}
\sum_{l=0}^{\left\lfloor\frac{n}{2}\right\rfloor}L_{h,n}^{l}(x)
&=\left(\left\lfloor\frac{n}{2}\right\rfloor+1\right)L_{h,n}(x) -
\sum_{i=0}^{\left\lfloor \frac{n}{2}\right\rfloor}i\frac{n}{n-i}\binom{n-i}{i}h^{n-2i}(x)
\end{align*}
From Lemma \ref{lem3} the Eq.(\ref{ec14})  is obtained.
\end{proof}

\section{Generating functions of the incomplete $h(x)$-Fibonacci and $h(x)$-Lucas polynomials}

In this section, we give the generating functions of incomplete $h(x)$-Fibonacci and $h(x)$-Lucas polynomials.

\begin{lemma}\label{lem4}(See \cite{PINTER}, p. 592). Let $\left\{s_n\right\}_{n=0}^{\infty}$ be a complex sequence satisfying the followin non-homogeneous recurrence relation:
\begin{align}
s_n=as_{n-1}+bs_{n-2}+r_n  \ \ (n>1),
\end{align}
where $a$ and $b$ are complex numbers and $\left\{r_n\right\}$ is a given complex sequence. Then the generating function $U(t)$ of the sequence $\left\{s_n\right\}$ is
\begin{align*}
U(t)=\frac{G(t)+s_0-r_0+(s_1-s_0a-r_1)t}{1-at-bt^2}
\end{align*}
where $G(t)$ denotes the generating function of $\left\{r_n\right\}$.
\end{lemma}

\begin{theorem}\label{teo1}
The generating function of the incomplete $h(x)$-Fibonacci  polynomials  $F_{h,n}^{l}(x)$ is given by
\begin{multline}
R_{h,l}(x)=\sum_{i=0}^{\infty}F_{h,i}^{l}(x)t^i\\
=t^{2l+1}\left[F_{h,2l+1}(x)
+ \left(F_{h,2l+2}(x)-h(x)F_{h,2l+1}(x)\right)t-\frac{t^2}{(1-h(x)t)^{l+1}}\right]\left[1-h(x)t-t^2\right]^{-1}
\end{multline}
\end{theorem}
\begin{proof}
Let $l$ be a fixed positive integer. From (\ref{ec1}) and (\ref{ec15}), $F_{h,n}^{l}(x)=0$ for $0\leq n < 2l+1$, $F_{h,2l+1}^{l}(x)=F_{h,2l+1}(x)$, and  $F_{h,2l+2}^{l}(x)=F_{h,2l+2}(x)$, and that
\begin{align}
F_{h,n}^{l}(x)=h(x)F_{h,n-1}^{l}(x) + F_{h,n-2}^{l}(x) - \binom{n-3-l}{l}h^{n-3-2l}(x).
\end{align}
Now let
\begin{align*}
s_0=F_{h,2l+1}^{l}(x),  s_1=F_{h,2l+2}^{l}(x), \ \text{and} \  s_n=F_{h,n+2l+1}^{l}(x).
\end{align*}
Also let
\begin{align*}
r_0=r_1=0 \ \text{and} \   r_n=\binom{n+l-1}{n-2}h^{n-2}(x).
\end{align*}
the generating function of the sequence  $\left\{r_n\right\}$ is   $G(t)=\frac{t^2}{(1-h(x)t)^{l+1}}$  (See  \cite[p. 355]{sriv}). Thus, from Lemma \ref{lem4}, we get the generating function $R_{h,l}(x)$  of sequence $\left\{s_n\right\}$.
\end{proof}

\begin{theorem}\label{teo2}
The generating function of the incomplete $h(x)$-Lucas  polynomials  $L_{h,n}^{l}(x)$ is given by
\begin{multline}
S_{h,l}(x)=\sum_{i=0}^{\infty}L_{h,i}^{l}(x)t^i\\
=t^{2l}\left[L_{h,2l}(x)
+ \left(L_{h,2l+1}(x)-h(x)L_{h,2l}(x)\right)t- \frac{t^2(2-t)}{(1-h(x)t)^{l+1}} \right]\left[1-h(x)t-t^2\right]^{-1}
\end{multline}
\end{theorem}
\begin{proof}
The proof of this theorem is similar to the proof of Theorem \ref{teo1}. Let $l$ be a fixed positive integer. From (\ref{ec2}) and (\ref{ec16}), $L_{h,n}^{l}(x)=0$ for $0\leq n < 2l$, $L_{h,2l}^{l}(x)=L_{h,2l}(x)$, and  $L_{h,2l+1}^{l}(x)=L_{h,2l+1}(x)$, and that
\begin{align}
L_{h,n}^{l}(x)=h(x)L_{h,n-1}^{l}(x) + L_{h,n-2}^{l}(x) - \frac{n-2}{n-2-l}\binom{n-2-l}{n-2-2l}h^{n-2-2l}(x)
\end{align}

Now let
\begin{align*}
s_0=L_{h,2l}^{l}(x),  s_1=L_{h,2l+1}^{l}(x), \ \text{and} \  s_n=L_{h,n+2l}^{l}(x).
\end{align*}
Also let
\begin{align*}
r_0=r_1=0 \ \text{and} \   r_n=\binom{n+2l-2}{n+l-2}h^{n+2l-2}(x).
\end{align*}
the generating function of the sequence  $\left\{r_n\right\}$ is   $G(t)=\frac{t^2(2-t)}{(1-h(x)t)^{l+1}}$  (See  \cite[p. 355]{sriv}). Thus, from Lemma \ref{lem4}, we get the generating function $S_{h,l}(x)$  of sequence $\left\{s_n\right\}$.
\end{proof}

\section{Acknowledgments}
The author would like to thank the anonymous referees for their helpful comments.  The author was partially supported by Universidad Sergio Arboleda under Grant no. USA-II-2012-14.

\end{document}